\documentclass[11pt]{amsart}
\usepackage{amssymb}
\usepackage{graphicx}
\usepackage{xcolor} 
\usepackage{tensor}
\usepackage{fullpage} 
\usepackage{amsmath}
\usepackage{amsthm}
\usepackage{verbatim}
\usepackage{hyperref}
\usepackage{array} 
\usepackage{enumitem}
\setlist[enumerate]{leftmargin=1.2em}
\setlist[itemize]{leftmargin=1.2em}
\usepackage{seqsplit,mathtools}
\mathtoolsset{showonlyrefs}

\definecolor{green}{rgb}{0,0.5,0} 

\newcommand{\Red}[1]{\begingroup\color{red} #1\endgroup} 

\newtheorem{theorem}{Theorem}[section]
\newtheorem{thm}{Theorem}[section]

\newtheorem{lem}[theorem]{Lemma}

\newtheorem{prop}[theorem]{Proposition}

\theoremstyle{definition}
\newtheorem{definition}[theorem]{Definition}

\theoremstyle{remark}
\newtheorem{remark}[theorem]{Remark}
\newtheorem{rem}[theorem]{Remark}

\numberwithin{equation}{section}

\numberwithin{equation}{section}
\newcommand{\nrm}[1]{\Vert#1\Vert}

\newcommand{\br}[1]{\overline{#1}}

\newcommand{\nnrm}[1]{{\vert\kern-0.25ex\vert\kern-0.25ex\vert #1 
		\vert\kern-0.25ex\vert\kern-0.25ex\vert}}

\newcommand{\dist}{\mathrm{dist}\,}

\newcommand{\supp}{{\mathrm{supp}}\,}

\newcommand{\lap}{\Delta}

\newcommand{\rd}{\partial}
\newcommand{\nb}{\nabla}

\newcommand{\ift}{\infty}

\newcommand{\alp}{\alpha}
\newcommand{\bt}{\beta}

\newcommand{\dlt}{\delta}

\newcommand{\eps}{\epsilon}
\newcommand{\veps}{\varepsilon}

\newcommand{\lmb}{\lambda}

\newcommand{\tht}{\theta}

\newcommand{\omg}{\omega}
\newcommand{\Omg}{\Omega}

\newcommand{\bfe}{{\bf e}}

\newcommand{\bfu}{{\bf u}}
\newcommand{\bfv}{{\bf v}}

\newcommand{\bfx}{{\bf x}}

\newcommand{\bbR}{\mathbb R}
\newcommand{\bbS}{\mathbb S}

\begin{document}
	\bibliographystyle{plain}
	\title{Support growth of vorticity for bi-rotational Euler flows in high dimensions}

	\renewcommand{\thefootnote}{\fnsymbol{footnote}}
	\renewcommand{\thefootnote}{\arabic{footnote}}
	
	\author{In-Jee Jeong}
	\address{Department of Mathematical Sciences and RIM, Seoul National University, 1 Gwanak-ro, Gwanak-gu, Seoul 08826}
	\email{injee$ \_ $j@snu.ac.kr}
	
	\author{Deokwoo Lim}
	\address{The Research Institute of Basic Sciences, Seoul National University, 1 Gwanak-ro, Gwanak-gu, Seoul 08826, Republic of Korea.}
	\email{dwlim95@snu.ac.kr}

	\date\today
	\maketitle
	
	\begin{abstract}
		We study incompressible Euler equations in $\mathbb{R}^d$ with $d \ge 4$ under bi-rotational symmetry without swirl, which reduces the Euler equations to a scalar vorticity advection in the first quadrant. We show that patch type initial vorticities exhibit infinite growth of the support diameter. 
	\end{abstract}

	\medskip
	

	\section{Introduction}\label{sec:intro}

	\subsection{Bi-rotational symmetry}
	
	The incompressible Euler equations in $\bbR^d$ are given by
	\begin{equation}\label{eq:Eulereq}
		\left\{
		\begin{aligned}
			\rd_{t} \bfv +  \bfv \cdot\nb \bfv +\nb p&=0,\\
			\nb\cdot \bfv &=0, 
		\end{aligned}
		\right.
	\end{equation}
	where $ \bfv = (v^{1},\cdots,v^{d}) : [0,T)\times\bbR^{d}\to\bbR^{d} $ and $ p : [0,T)\times\bbR^{d}\to\bbR $. In this paper, we take $d \ge 4$ and study \eqref{eq:Eulereq} under the \textbf{bi-rotational symmetry}, using the bi-polar coordinates: we take two integers $n, m \ge 1$  where $d = n + m + 2\ge 4$, decompose $\bbR^d = \bbR^{n+1} \times \bbR^{m+1}$, and introduce two radii $r, s \ge 0$ by \begin{equation*}
		\begin{split}
			r^{2} = \sum_{i=1}^{n+1}x_{i}^{2}, \qquad s^{2} = \sum_{i=n+2}^{n+m+2}x_{i}^{2}. 
		\end{split}
	\end{equation*} We define two unit vectors in $\bbR^{d}$ by $\bfe^r = \frac1r(x_1,...,x_{n+1},0,...,0)$ and $\bfe^{s} = \frac1s(0,...,0,x_{n+2},...,x_{n+m+2})$. We shall refer to \textit{bi-rotational symmetric solutions without swirl} to be solutions of \eqref{eq:Eulereq} of the form 
	\begin{equation}\label{eq:birothighd}
		\bfv(t,\bfx) = u^{r}(t,r,s)\bfe^{r} + u^{s}(t,r,s)\bfe^{s}, \qquad \bfx = (x_{1},...,x_{n+m+2}), 
	\end{equation} where $u^{r}, u^{s}$ are scalar-valued. As long as the right-hand side of \eqref{eq:birothighd} satisfies some regularity and decay conditions at $t = 0$, we shall see that there is a unique corresponding local-in-time solution $\bfv(t,\bfx)$ to \eqref{eq:Eulereq}, which verifies the ansatz \eqref{eq:birothighd}. The term ``bi-rotational'' means that the velocity is invariant under all rotations of $\bbR^d$ which either fixes the first $n+1$ coordinates or the last $m+1$ coordinates, and ``without swirl'' means that $\bfv \in \mathrm{span}\{ \bfe^{r}, \bfe^{s} \}$ for every point in space-time. 
	
	\subsection{Vorticity formulation of the bi-rotational solutions}
	Recall that in $\bbR^d$ the vorticity is a two-tensor $\boldsymbol{\omg}$ defined from $\bfv$ by (cf. \cite{Chemin})
	\begin{equation}\label{eq:vorttensor}
		\boldsymbol{\omg} = (\omg^{i,j})_{1\leq i,j\leq d},\quad \omg^{i,j}=\rd_{j}v^{i}-\rd_{i}v^{j}.
	\end{equation} The corresponding stream function is defined component-wise by
	\begin{equation}\label{eq:streamftntensor}
		\Psi^{i,j}=\lap^{-1}\omg^{i,j}
	\end{equation} where $\lap^{-1}$ is the inverse Laplacian operator defined in $\bbR^{d}$. Under the bi-rotational symmetry assumption on the velocity \eqref{eq:birothighd}, there exist two scalar-valued functions of $r$ and $s$ satisfying 
	\begin{equation}\label{eq:vorttensca}
		\omg^{i,j}(\bfx)=-\frac{x_{i}x_{j}}{rs}w(r,s),\qquad \Psi^{i,j}(\bfx)=-\frac{x_{i}x_{j}}{rs}\psi(r,s)
	\end{equation} for all $1 \le i \le n+1, n+2 \le j \le n+m+2$ and $\bfx \in \bbR^{d}$. The proof is given in Section \ref{sec:computation}. With an abuse of notation, we shall refer to $w$ and $\psi$ as the (scalar) vorticity and stream function, respectively. Furthermore, the Euler equations reduce to a scalar evolution equation for $w$: 
	\begin{equation}\label{eq:vorteq}
		\rd_{t}w+(u^{r}\rd_{r}+u^{s}\rd_{s})w=\bigg(n\frac{u^{r}}{r}+m\frac{u^{s}}{s}\bigg)w,
	\end{equation}
	or,
	\begin{equation}\label{eq:vorticityeq}
		\rd_{t}\left(\frac{w}{r^{n}s^{m}}\right)+(u^{r}\rd_{r}+u^{s}\rd_{s})\left(\frac{w}{r^{n}s^{m}}\right)=0,
	\end{equation} where $w : [0,T) \times \Pi \to \bbR$ with $\Pi = \left\{ (r, s) \, : \, r, s \ge 0 \right\}$. The bi-rotational velocity $\bfu := (u^{r}, u^{s})$ satisfies the divergence-free condition (defined in $\bbR^d$) \begin{equation}\label{eq:divergencefree}
		\rd_{r}u^{r}+\frac{n}{r}u^{r}+\rd_{s}u^{s}+\frac{m}{s}u^{s}=0
	\end{equation} as well as $w = \rd_r u^{s} - \rd_s u^{r}$. As such, using the stream function $\psi$ which is related with $w$ by  \begin{equation}\label{eq:psi-w-relation}
		\left\{
		\begin{aligned}
			\bigg(\rd_{r}^{2}+\frac{n}{r}\rd_{r}-\frac{n}{r^{2}}+\rd_{s}^{2}+\frac{m}{s}\rd_{s}-\frac{m}{s^{2}}\bigg) \psi = w \quad &\mbox{in} \quad \Pi , \\
			\psi = 0 \quad &\mbox{on} \quad \partial\Pi = \left\{ (r,s) \, : \, r = 0 \mbox{ or } s = 0 \right\}, 
		\end{aligned}
		\right.
	\end{equation} we can recover $\bfu$ from $\psi$ from 
	\begin{equation}\label{eq:psicurl}
		u^{r}=-\frac{\rd_{s}(s^{m}\psi)}{s^{m}}=-\bigg(\frac{m}{s}\psi+\rd_{s}\psi\bigg),\quad u^{s}=\frac{\rd_{r}(r^{n}\psi)}{r^{n}}=\frac{n}{r}\psi+\rd_{r}\psi.
	\end{equation} The details of these 
	computations are in the Section \ref{sec:computation}.


	

	\subsection{Main result}

	The form of the equation \eqref{eq:vorticityeq} shows that $w/(r^n s^m)$ is simply being transported by $\bfu$. This allows us to consider the \textit{patch-type} solutions, where $w/(r^n s^m)$ is the indicator function of a (time-dependent) set in $\Pi$. 	We are interested in how the support size of $w$ evolves for patch-type solutions of \eqref{eq:vorticityeq}. Assuming that we are given a solution $w\in L^{\ift}(0,T;L^{\ift}(\bbR^{d}))$ which is compactly supported in $r$ and $s$ for each $t \in [0,T]$, we define  \begin{equation}\label{eq:def-RSL}
		\begin{split}
			R(t)&:=\sup \lbrace r : (r,s)\in \supp w(t)\rbrace,\quad S(t):=\sup \lbrace s : (r,s)\in \supp w(t)\rbrace,\quad L(t):=R(t)+S(t).
		\end{split}
	\end{equation}
	The following result shows local well-posedness for the patch-type data, when the support is initially separated from the axes. 
	
	\begin{thm}\label{thm:0}
		Take a bounded open set  $\Omg_{0}\subset \Pi$ which is separated from the boundary $\rd\Pi$ of the domain; that is, $\overline{\Omg_{0}} \cap \partial\Pi = \emptyset$. Then, consider the  patch-type initial datum 
		\begin{equation}\label{eq:patch}
			w_{0}(r,s)=r^{n}s^{m} \mathbf{1}_{\Omg_{0}}(r,s). 
		\end{equation} 
		There exist a time $T>0$ and  a local unique solution of the Euler equations corresponding to \eqref{eq:patch}, which also takes the form
		\begin{equation}\label{eq:patchsol}
			w(t,r,s)=r^{n}s^{m}\mathbf{1}_{\Omg_{t}}(r,s),\quad t\in[0,T
			), 
		\end{equation}
		where $\Omg_{t}\in \Pi$ is the image of the set $\Omg_{0}$ along the flow generated by $\bfu$ in $\Pi$. Furthermore, the solution can be extended past $T$ if and only if  
		\begin{equation}\label{eq:blowupcrit}
			\sup_{t\in[0,T)}L(t)
			<+\ift,
		\end{equation}  where $L(t)$ is defined in \eqref{eq:def-RSL}. 
	\end{thm}
	\begin{rem}
		We note that for $t < T$, $w \in L^\infty(\Pi)$ and this implies $\boldsymbol{\omg} \in L^\infty(\bbR^d)$. Moreover, $\boldsymbol{\omg}$ is compactly supported in space and this guarantees that the corresponding velocity $\bfv$ is log-Lipschitz in $\bbR^d$. In turn, $\bfu = (u^r, u^s)$ is log-Lipschitz in $\Pi$ and $w/(r^n s^m)$ is being transported by the flow map generated by $\bfu$ in $\Pi$.  
	\end{rem}
	
	Given $\Omg_{0}$ satisfying the assumptions of Theorem \ref{thm:0}, we may define $T_{max} = T_{max}(\Omg_{0}) \in (0,+\infty]$ to be the maximal lifespan of the solution with initial datum \eqref{eq:patch}. The following result shows that even when $T_{max} = +\infty$, we always have $L(t) \to \infty$ along a sequence of time moments.  
	
	\begin{thm}\label{thm:1} 
		Let $w(t)$ be the local-in-time solution from Theorem \ref{thm:0} with the maximal time of existence $T_{\text{max}}>0$. Then the support of $w(t)$ satisfies
		\begin{equation}\label{eq:patchstretching}
			\sup_{t\in[0,T_{\text{max}})} L(t)
			= +\ift.
		\end{equation} 
	\end{thm}
	
	\begin{rem}
		Using the time reversibility of the Euler equations, we can obtain the same conclusion when we solve the initial datum backwards in time. This corresponds to simply flipping the sign of the initial datum and solve forwards in time.  In a similar vein, the conclusion also holds when $\omg_{0} = \lambda r^n s^m \mathbf{1}_{\Omg_{0}}(r,s)$ for any constant $\lambda \ne 0$. 
	\end{rem}
	
	Unfortunately, the proof of Theorem \ref{thm:1} is largely based on a contradiction argument and it does not provide detailed information about the shape of the set $\Omg_{t}$ as $t\to T_{max}$. For instance, one can ask what happens to the diameter, perimeter, curvature of the patch. Furthermore, we were not able to show that $\nrm{w(t,\cdot)}_{L^\infty}\to+\infty$ as $t\to\infty$ when $T_{max} = +\infty$. This is because along the flow trajectories, $r^n s^m$ can remain bounded (or even decay) on $\Omg_{t}$. We leave these issues as open problems. 
	
	\subsection{Discussion of the main results}\label{ssec:literature}
	
	While it is known that the incompressible Euler equations in high dimensions can exhibit all kinds of dynamics (\cite{TaoUniv,DE}), in general it is very difficult to analyze the solutions, as \eqref{eq:Eulereq} is a system of $d+1$ scalar functions defined in $\bbR^d$. Assuming appropriate rotational symmetries on the high dimensional Euler solutions, one may reduce both the number of dependent and independent variables, while retaining some interesting dynamics. The most widely studied case is when the solution is assumed to be invariant under all rotations fixing an axis, and there have been several recent progress on high-dimensional axisymmetric Euler equations (\cite{CJLglobal22, Miller, GMT2023, Limglobal23, LJ_optimal}), concerning well-posedness, confinement, singularity formation, and so on. It seems that for higher dimensions, there is a better chance for finite-time singularity formation. 
	
	On the other hand, dynamics under the bi-rotational symmetry  \eqref{eq:birothighd} seems to have been less studied. The bi-rotational equation \eqref{eq:vorticityeq} appears in the works of Khesin--Yang \cite{KhYa} and Yang \cite{Yang} (they use the terminology $\bbS^n \times \bbS^m$-symmetric solutions), where they noted that for any $n,m\ge 1$, it is a special case of the so-called lake equations with the depth function $d(r,s) = r^n s^m$. The works \cite{KhYa,Yang} study the point vortex motion evolving by the skew-mean curvature
	flow in the $(r,s)$-plane, which corresponds to ``vortex membranes'' of codimension two. The resulting formal ODE system of the point vortex motion blows up in finite time for $d \ge 5$, but it seems very difficult to justify the ODE using smooth solutions. One could replace the point vortex by a highly concentrated patch and study the resulting patch dynamics which is well-defined by Theorem \ref{thm:0}. 
	
	Recently, \cite{CJL_gwpbi} studied the bi-rotational equation in the simplest case $n = m = 1$ (i.e. $d = 4$) and obtained global well-posedness of $w$ in the Yudovich-type class \begin{equation*}
		\begin{split}
			(1+r+s) \frac{w}{rs} \in L^{4,1}(\bbR^4) \quad\mbox{ and } \quad w \in (L^{\infty} \cap L^{4,1})(\bbR^4)
		\end{split}
	\end{equation*} where $L^{p,q}$ denotes the Lorenz space. The condition $L^{4,1}$ is inspired by the work of Danchin \cite{Danaxi} in 3D axisymmetric flows and is probably sharp. This result implies in particular that sufficiently smooth and decaying initial data $\bfv_{0}$ has a global-in-time unique smooth solution under bi-rotational symmetry in $\bbR^4$. Global well-posedness of bi-rotational solutions seems to be open when $d \ge 5$.  
	
	For bi-rotational symmetric solutions, when the scalar vorticity is compactly supported, blowup is possible only if the support of the vorticity diverges to infinity in finite time (see Section \ref{sec:lwp} for details). This was our main motivation behind Theorem \ref{thm:1}. For high dimensional axisymmetric solutions, there are recent works (\cite{CJ-axi,GMT2023}) providing some quantitative lower bounds on the support growth in the $r$-direction, where $r^2 ={x_1^2+...+x_{d-1}^2}$ and $z = x_{d}$. Gustafson--Miller--Tsai proves in \cite{GMT2023} the lower bounds \begin{equation*}
		\begin{split}
			\iint_{ r,z \in [0,\infty) } r^{d-1} w(t,r,z) drdz \gtrsim \begin{cases}
				(1+t)^{3/4-\eps}, \quad & d = 3, \\
				(1+t)^{2/3-\eps}, \quad & d = 4 \\ 
				(1+t)^{d/(d^2-2d-2)-\eps}, \quad & d \ge 5
			\end{cases}
		\end{split}
	\end{equation*} where the implicit constant depends on $w_{0}, d, \eps>0$, assuming that the solution does not blowup until time $t$. This in particular gives a quantitative growth rate for the support radius in $r$. On the other hand, our growth result is not quantitative and works only for patch-type data. Still, it is the first infinite growth result for bi-rotational solutions. In this case, the difficulty comes from understanding the kernel of the velocity in terms of the scalar vorticity; the proof of \cite{GMT2023} utilizes the representation of the velocity kernel in terms of certain elliptic integrals (cf. \cite{FeSv}), which is not available in the bi-rotational case.

	

	\subsection*{Organization of the paper} In Section \ref{sec:lwp}, we derive the relation between the stream function, velocity and vorticity under the bi-rotational symmetry. We prove Theorems \ref{thm:0} and \ref{thm:1} in Section \ref{sec:lwp} and Section \ref{sec:growth}, respectively. 
	
	\subsection*{Acknowledgments}
	I.-J. Jeong was supported by the NRF grant from the Korea government (MSIT), No. 2022R1C1C1011051, RS-2024-00406821. D. Lim was supported by the National Research Foundation of Korea grant RS-2024-00350427.
	
	\section{Computation of the vorticity and the stream function}\label{sec:computation}
	
	In this 
	section, we present computations of the vorticity and the stream function under the bi-rotational symmetry and no-swirl condition on the velocity \eqref{eq:birothighd} to show existence 
	of the scalar vorticity $w$ and the scalar stream function $\psi$ satisfying \eqref{eq:vorttensca}.

	\subsection{Computation of the vorticity}
	
	Under the bi-rotational symmetry and no-swirl condition \eqref{eq:birothighd}, we have
	\begin{equation}\label{eq:uiur}
		v^{i}
		=\begin{cases}
			\frac{\rd r}{\rd x_{i}}u^{r}=\frac{x_{i}}{r}u^{r},&\quad i=1,\cdots,n+1,\\
			\frac{\rd s}{\rd x_{i}}u^{s}=\frac{x_{i}}{s}u^{s},&\quad i=n+2,\cdots,n+m+2.
		\end{cases}
	\end{equation}
	For any $ 1\leq i, j \leq n+1$, we have
	\begin{equation}
		\begin{split}
			\omg^{i,j}&=\rd_{x_{j}}v^{i}
			-\rd_{x_{i}}v^{j}
			\\
			&=\bigg[\frac{\rd r}{\rd x_{j}}\rd_{r}+\sum_{l=1}^{n}\frac{\rd \tht_{l}}{\rd x_{j}}\rd_{\tht_{l}}\bigg]\bigg(\frac{\rd r}{\rd x_{i}}u^{r}\bigg)-\bigg[\frac{\rd r}{\rd x_{i}}\rd_{r}+\sum_{l=1}^{n}\frac{\rd \tht_{l}}{\rd x_{i}}\rd_{\tht_{l}}\bigg]\bigg(\frac{\rd r}{\rd x_{j}}u^{r}\bigg)
			=0.
		\end{split}
	\end{equation}
	Similarly, for any $n+2\leq i, j \leq n+m+2$, we get
	\begin{equation}
		\begin{split}
			\omg^{i,j}=\bigg[\frac{\rd s}{\rd x_{j}}\rd_{s}+\sum_{l=1}^{m}\frac{\rd \phi_{l}}{\rd x_{j}}\rd_{\phi_{l}}\bigg]\bigg(\frac{\rd s}{\rd x_{i}}u^{s}\bigg)-\bigg[\frac{\rd s}{\rd x_{i}}\rd_{s}+\sum_{l=1}^{m}\frac{\rd \phi_{l}}{\rd x_{i}}\rd_{\phi_{l}}\bigg]\bigg(\frac{\rd s}{\rd x_{j}}u^{s}\bigg)
			=0.
		\end{split}
	\end{equation}
	On the other hand, for any $1\leq i \leq n+1$ and $ n+2 \leq j \leq n+m+2 $, we have
	\begin{equation}
		\begin{split}
			\omg^{i,j}
			=\bigg[\frac{\rd s}{\rd x_{j}}\rd_{s}+\sum_{l=1}^{m}\frac{\rd \phi_{l}}{\rd x_{j}}\rd_{\phi_{l}}\bigg]\bigg(\frac{\rd r}{\rd x_{i}}u^{r}\bigg)-\bigg[\frac{\rd r}{\rd x_{i}}\rd_{r}+\sum_{l=1}^{n}\frac{\rd \tht_{l}}{\rd x_{i}}\rd_{\tht_{l}}\bigg]\bigg(\frac{\rd s}{\rd x_{j}}u^{s}\bigg) =-\frac{x_{i}x_{j}}{rs}(\rd_{r}u^{s}-\rd_{s}u^{r}),
		\end{split}
	\end{equation}
	and $\omg^{j,i}=-\omg^{i,j}$. We denote $w:=\rd_{r}u^{s}-\rd_{s}u^{r}$.
	
	\subsection{Computation of the stream function}
	
	Recall that the divergence-free condition \eqref{eq:divergencefree} implies the existence of the scalar function $\psi$ that satisfies
	\begin{equation}\label{eq:uruspsi}
		u^{r}=-\bigg(\frac{m}{s}\psi+\rd_{s}\psi\bigg),\quad u^{s}=\frac{n}{r}\psi+\rd_{r}\psi,
	\end{equation}
	which solves the equation \eqref{eq:psi-w-relation}:
	\begin{equation*}
		\left\{
		\begin{aligned}
			\bigg(\rd_{r}^{2}+\frac{n}{r}\rd_{r}-\frac{n}{r^{2}}+\rd_{s}^{2}+\frac{m}{s}\rd_{s}-\frac{m}{s^{2}}\bigg) \psi = w \quad &\mbox{in} \quad \Pi , \\
			\psi = 0 \quad &\mbox{on} \quad \partial\Pi = \left\{ (r,s) \, : \, r = 0 \mbox{ or } s = 0 \right\}, 
		\end{aligned}
		\right.
	\end{equation*}
	Note that 	we can rewrite the equation \eqref{eq:psi-w-relation} as
		\begin{equation}\label{eq:rslapzt}
			\lap_{\bbR^{n+3}\times\bbR^{m+3}}\bigg(\frac{\psi}{rs}\bigg)=\frac{w}{rs},
		\end{equation}
		where $\lap_{\bbR^{n+3}\times\bbR^{m+3}}$ is the Laplacian in the space $\bbR^{d+4}=\bbR^{n+3}\times\bbR^{m+3}$. Then 
		we can recover $ \psi $ as
		\begin{equation}
			\psi
			=rs\lap_{\bbR^{n+3}\times\bbR^{m+3}}^{-1}\bigg[\frac{w}{rs}\bigg].
		\end{equation}
		Due to the rotational invariance of $ \lap_{\bbR^{n+3}\times\bbR^{m+3}} $ and the dependence of $w$ only on $r$ and $s$, the same dependence holds for $\psi$ as well. 
		
		Now observe that for any $1\leq i\leq n+1,\; n+2\leq j\leq n+m+2$, taking the Laplacian $\lap$ in $\bbR^{d}$ on the term $(x_{i}x_{j}\psi)/(rs)$ gives us
		\begin{equation}
			\begin{split}
				\lap\bigg(\frac{x_{i}x_{j}}{rs}\psi\bigg)&=\lap\bigg(\frac{x_{i}x_{j}}{rs}\bigg)\psi+2\nb\bigg(\frac{x_{i}x_{j}}{rs}\bigg)\cdot\nb\psi+\frac{x_{i}x_{j}}{rs}\lap\psi\\
				&=\frac{x_{i}x_{j}}{rs}\bigg[-\bigg(\frac{n}{r^{2}}+\frac{m}{s^{2}}\bigg)+0+\bigg(\rd_{r}^{2}+\frac{n}{r}\rd_{r}+\rd_{s}^{2}+\frac{m}{s}\rd_{s}\bigg)\bigg]\psi=\frac{x_{i}x_{j}}{rs}w=-\omg^{i,j}.
			\end{split}
		\end{equation}
		We used the equation \eqref{eq:psi-w-relation} in the third equality. From this, for any $\bfx\in \bbR^{d}$, $r,s\geq0$, we can write the stream function \eqref{eq:streamftntensor} in the following form:
		\begin{equation}\label{eq:streamftnrel}
			\Psi^{i,j}(\bfx)=\lap^{-1}\omg^{i,j}(\bfx)=-\frac{x_{i}x_{j}}{rs}\psi(r,s),
		\end{equation}
		and $\Psi^{j,i}=-\Psi^{i,j}$ when 
		$1\leq i\leq n+1,\; n+2\leq j\leq n+m+2$, and $\Psi^{i,j}\equiv0$ when $1\leq i,j\leq n+1$ or $n+2\leq i,j \leq n+m+2$. In particular, by taking $i=1$, $j=n+2$, and $\tht_{1}=\cdots=\tht_{n}=\phi_{1}=\cdots=\phi_{m}=0$ on the above, we can obtain the explicit form of the scalar stream function $\psi$ as
		\begin{equation}\label{eq:streamftn}
			\begin{split}
				&\psi(r,s)=-\lap^{-1}\omg^{1,n+2}(\bfx)\big|_{\tht_{1}=\cdots=\tht_{n}=\phi_{1}=\cdots=\phi_{m}=0}\\
				&\;=c_{d}\int_{\bbR^{d}}\frac{1}{(r^{2}+\br{r}^{2}-2ry_{1}+s^{2}+\br{s}^{2}-2sy_{n+2})^{d/2-1}}\frac{y_{1}y_{n+2}}{\br{r}\br{s}}w(\br{r},\br{s})dy\\
				&\;=c_{d}\iint_{\Pi}\int_{0}^{2\pi}
				\int_{[0,\pi]^{m-1}}\int_{0}^{2\pi}
				\int_{[0,\pi]^{n-1}}\frac{1}{(r^{2}+\br{r}^{2}-2r\br{r}\cos\br{\tht}_{1}+s^{2}+\br{s}^{2}-2s\br{s}\cos\br{\phi}_{1})^{d/2-1}}\\
				&\quad
				\cdot\cos\br{\tht}_{1}\cos\br{\phi}_{1}w(\br{r},\br{s})
				\bigg(\br{r}^{n}\br{s}^{m}\prod_{i=1}^{n-1}\sin^{n-i}\br{\tht}_{i}\prod_{j=1}^{m-1}\sin^{m-j}\br{\phi}_{j}\bigg)d\br{\tht}_{1}\cdots d\br{\tht}_{n-1}d\br{\tht}_{n}d\br{\phi}_{1}\cdots d\br{\phi}_{m-1}d\br{\phi}_{m}d\br{r}d\br{s}
				\\
				&\;=c_{d}\iint_{\Pi}\int_{0}^{\pi}\int_{0}^{\pi}\frac{\br{r}^{n}\br{s}^{m}\sin^{n-1}\br{\tht}_{1}\cos\br{\tht}_{1}\sin^{m-1}\br{\phi}_{1}\cos\br{\phi}_{1}}{(r^{2}+\br{r}^{2}-2r\br{r}\cos\br{\tht}_{1}+s^{2}+\br{s}^{2}-2s\br{s}\cos\br{\phi}_{1})^{d/2-1}}w(\br{r},\br{s})d\br{\tht}_{1}d\br{\phi}_{1}d\br{r}d\br{s},
			\end{split}
		\end{equation}
		where $c_{d}>0$ denotes a constant that depends only on $d$ which may change line by line. 
		Indeed, one can use the Biot--Savart law in the Cartesian coordinates (cf. \cite{Chemin})
		\begin{equation}\label{eq:BSlawCartesian}
			v^{i}
			=\sum_{j=1}^{d}\rd_{x_{j}}\Psi^{i,j}
		\end{equation}
		to compute $v^{1}$ 
		and $v^{n+2}$, 
		and use the 
		change of coordinates 
		\begin{equation}
			\frac{u^{r}}{r}=\frac{v^{1}}{x_{1}},\quad \frac{u^{s}}{s}=\frac{v^{n+2}}{x_{n+2}},
		\end{equation}
		to check that $\psi$ satisfies the relation \eqref{eq:uruspsi}:
		\begin{equation}
			\begin{split}
				u^{r}&=\frac{r}{x_{1}}v^{1}=-\frac{r}{x_{1}}\bigg[\cos\tht_{1}\bigg(\rd_{s}\psi+\frac{m}{s}\psi\bigg)\bigg]=-\bigg(\rd_{s}\psi+\frac{m}{s}\psi\bigg), \\
				u^{s} &=\frac{s}{x_{n+2}}v^{n+2}=\frac{s}{x_{n+2}}\bigg[\cos\phi_{1}\bigg(\rd_{r}\psi+\frac{n}{r}\psi\bigg)\bigg]=\rd_{r}\psi+\frac{n}{r}\psi.
			\end{split}
		\end{equation}

	\begin{remark}
		The term $s^{-1}\psi$ from \eqref{eq:psicurl} is non-singular at $s=0$ and non-negative; one can check by a direct computation that 
		$$ \frac{\psi(r,s)}{s}\bigg|_{s=0}=\rd_{s}\psi(r,0). $$
	\end{remark}
	
	\section{Local well-posedness of the patch-type solution and blowup criterion} \label{sec:lwp}
	
	We provide a more general local well-posedness result, which implies Theorem \ref{thm:0}. For simplicity, we define the norm \begin{equation}\label{eq:X}
		\begin{split}
			\nrm{w}_{X} := \nrm{w}_{L^{\ift}(\bbR^{d})} + \bigg\| (1+r^{n}+s^{m})\frac{w}{r^{n}s^{m}}\bigg\|
			_{ L^{\ift}(\bbR^{d})}. 
		\end{split}
	\end{equation}
	
	\begin{prop}\label{prop:lwp}
		Assume that the initial datum $w_{0}(r,s)\in X$ is compactly supported. Then there exist $T>0$ and  a unique solution  of \eqref{eq:vorticityeq} that satisfies $w \in L^{\infty}(0,T;X)$.  {Furthermore, the solution $w$ blows up in the norm $X$ at a time $0<T^{\ast}<+\ift$ if and only if we have
			\begin{equation}\label{eq:blowupL}
				\int_{0}^{T^{\ast}}{L(t)^{d-1}}
				dt=+\ift
			\end{equation} where $L(t)$ is defined in \eqref{eq:def-RSL}.}
	\end{prop}
	
	\begin{proof} 
		Note that from the evolution equation \eqref{eq:vorticityeq} of $w/(r^{n}s^{m})$
		, we have
		\begin{equation}
			\bigg\| \frac{w(t)}{r^{n}s^{m}}\bigg\|_{L^{\ift}(\bbR^{d})}=\bigg\| \frac{w_{0}}{r^{n}s^{m}}\bigg\|_{L^{\ift}(\bbR^{d})}.
		\end{equation}
		Using this, we can get the following upper bound of the vorticity maximum:
		\begin{equation}
			\begin{split}
				\nrm{w(t)}_{L^{\ift}(\bbR^{d})}&\leq R(t)^{n}S(t)^{m}\bigg\| \frac{w(t)}{r^{n}s^{m}}\bigg\|_{L^{\ift}(\bbR^{d})}=R(t)^{n}S(t)^{m}\bigg\| \frac{w_{0}}{r^{n}s^{m}}\bigg\|_{L^{\ift}(\bbR^{d})} \leq
				L(t)^{n+m}\bigg\| \frac{w_{0}}{r^{n}s^{m}}\bigg\|_{L^{\ift}(\bbR^{d})}.
			\end{split}
		\end{equation}
		Here, we use the Biot--Savart law in Cartesian coordinates
		\begin{equation}
			v^{i}
			(t,x)=\sum_{j=1}^{d}\rd_{x_{j}}\Psi^{i,j}(t,x),\quad i=1, \cdots, d,\quad t\geq0,\quad x\in \bbR^{d},
		\end{equation}
		where
		\begin{equation}
			\Psi^{i,j}(t,x)=\lap^{-1}\omg^{i,j}(t,x)=c_{d}\int_{\bbR^{d}}\frac{1}{|x-y|^{d-2}}\omg^{i,j}(t,y)dy,\quad t\geq0,\quad x\in \bbR^{d}.
		\end{equation}
		Then note that we have
		\begin{equation}
			|\rd_{x_{j}}\Psi^{i,j}(t,x)|\lesssim_{d}
			\int_{\bbR^{d}}\frac{1}{|x-y|^{d-1}}|\omg^{i,j}(t,y)|dy,\quad t\geq0, \quad x\in \bbR^{d}.
		\end{equation}
		Also, note that the vorticity $\omg^{i,j}(t,\cdot)$ is supported in the closed ball $\br{B_{L(t)}(0)}=\lbrace x\in \bbR^{d} : |x|\leq L(t)\rbrace$. 
		Due to this, the maximum of the velocity is attained on the ball $\br{B_{L(t)}(0)}$. Then for any $|x|\leq L(t)$, we have
		\begin{equation}\label{eq:uLd-1}
			\begin{split}
				|\rd_{x_{j}}\Psi^{i,j}(t,x)|&\lesssim_{d}
				\int_{\bbR^{d}}\frac{1}{|x-y|^{d-1}}|\omg^{i,j}(t,y)|dy=
				\int_{B_{L(t)}(0)}\frac{1}{|x-y|^{d-1}}|\omg^{i,j}(t,y)|dy\\
				&\leq 
				\nrm{\omg^{i,j}(t)}_{L^{\ift}(B_{L(t)}(0))}\int_{B_{L(t)}(0)}\frac{1}{|x-y|^{d-1}}dy
				\lesssim_{d} 
				L\nrm{\omg^{i,j}(t)}_{L^{\ift}(B_{L(t)}(0))}\\
				&\leq 
				L\nrm{w(t)}_{L^{\ift}(B_{L(t)}(0))}\leq 
				LR^{n}S^{m}\bigg\| \frac{w(t)}{r^{n}s^{m}}\bigg\|_{L^{\ift}(B_{L(t)}(0))} \leq L^{d-1}\bigg\| \frac{w_{0}}{r^{n}s^{m}}\bigg\|_{L^{\ift}(\bbR^{d})}.
			\end{split}
		\end{equation}
		We used the relation $|x-y|\leq 2L(t)$ in the third inequality. From this, we obtain
		\begin{equation}\label{eq:dotLu}
			\dot{L}(t)\leq \nrm{u(t)}_{L^{\ift}(\bbR^{d})}\lesssim_{d,w_{0}}L(t)^{d-1}.
		\end{equation}
		This gives us
		\begin{equation}
			L(t)\lesssim_{d,w_{0}} 
			\frac{1}{[\frac{1}{
					L(0)^{d-2}}-C_{d}t]^{1/(d-2)}},\quad 0\leq t<\frac{1}{C_{d}
				L(0)^{d-2}}.
		\end{equation}
		Taking $T=\frac{1}{C_{d}
			L(0)^{d-2}}$, we have
		\begin{equation}
			\nrm{w(t)}_{L^{\ift}(\bbR^{d})}\lesssim_{w_{0}}L(t)^{d-2}\lesssim_{d,w_{0}}\frac{1}{C_{d}(T
				-t)},\quad 0\leq t<T.
		\end{equation}
		In addition, note that the terms $w/r^{n}$ and $w/s^{m}$ 
		satisfy the following evolution equations:
		\begin{align}
			\rd_{t}\frac{w}{r^{n}}+(u^{r}\rd_{r}+u^{s}\rd_{s})\frac{w}{r^{n}}&=m\frac{u^{s}}{s}\frac{w}{r^{n}},\label{eq:rnw}\\
			\rd_{t}\frac{w}{s^{m}}+(u^{r}\rd_{r}+u^{s}\rd_{s})\frac{w}{s^{m}}&=n\frac{u^{r}}{r}\frac{w}{s^{m}}.\label{eq:smw}
		\end{align}
		Here, we use the fact that we can interpolate the right-hand sides of equations \eqref{eq:rnw} and \eqref{eq:smw} as
		\begin{align}
			m\frac{u^{s}}{s}\frac{w}{r^{n}}&= mu^{r}\bigg(\frac{w}{r^{n}}\bigg)^{1-1/m}\bigg(\frac{w}{r^{n}s^{m}}\bigg)^{1/m}, \quad 
			n\frac{u^{r}}{r}\frac{w}{s^{m}} =nu^{s}\bigg(\frac{w}{s^{m}}\bigg)^{1-1/n}\bigg(\frac{w}{r^{n}s^{m}}\bigg)^{1/n}.
		\end{align}
		Using this, we have
		\begin{equation}\label{eq:r-nw}
			\begin{split}
				\frac{d}{dt}\bigg\|\frac{w}{r^{n}}\bigg\|_{L^{\ift}(\bbR^{d})}&\leq m\nrm{u^{s}}_{L^{\ift}(\bbR^{d})}\bigg\|\frac{w}{r^{n}}\bigg\|_{L^{\ift}(\bbR^{d})}^{1-1/m}\bigg\|\frac{w}{r^{n}s^{m}}\bigg\|_{L^{\ift}(\bbR^{d})}^{1/m} \lesssim mL^{d-1}\bigg\|\frac{w_{0}}{r^{n}s^{m}}\bigg\|_{L^{\ift}(\bbR^{d})}^{1+1/m}\bigg\|\frac{w}{r^{n}}\bigg\|_{L^{\ift}(\bbR^{d})}^{1-1/m},
			\end{split}
		\end{equation}
		and likewise,
		\begin{equation}\label{eq:s-mw}
			\begin{split}
				\frac{d}{dt}\bigg\|\frac{w}{s^{m}}\bigg\|_{L^{\ift}(\bbR^{d})}
				&\lesssim nL^{d-1}\bigg\|\frac{w_{0}}{r^{n}s^{m}}\bigg\|_{L^{\ift}(\bbR^{d})}^{1+1/n}\bigg\|\frac{w}{s^{m}}\bigg\|_{L^{\ift}(\bbR^{d})}^{1-1/n}.
			\end{split}
		\end{equation}
		Solving the above differential inequalities, we get
		\begin{align}
			\bigg\|\frac{w(t)}{r^{n}}\bigg\|_{L^{\ift}(\bbR^{d})}&\lesssim_{d,w_{0}}\bigg(1+\int_{0}^{t}L(\tau)^{d-1}d\tau\bigg)^{m},\quad 0\leq t<T,\label{eq:r-nwsol}\\
			\bigg\|\frac{w(t)}{s^{m}}\bigg\|_{L^{\ift}(\bbR^{d})}&\lesssim_{d,w_{0}}\bigg(1+\int_{0}^{t}L(\tau)^{d-1}d\tau\bigg)^{n},\quad 0\leq t<T.\label{eq:s-mwsol}
		\end{align}
		This finishes the proof of a priori estimate of the solution $w$. From this, we can obtain the existence of a solution via standard smoothing process (\cite{MP,Danaxi}).
		
		\medskip
		
		To show 
		the blowup criterion \eqref{eq:blowupL}, note that 
		from the a priori estimate of $\nrm{w}_{L^{\ift}(\bbR^{d})}$, we can get the following differential inequality:
		\begin{align}
			\frac{d}{dt}\nrm{w}_{L^{\ift}(\bbR^{d})}&\leq n\nrm{u^{r}}_{L^{\ift}(\bbR^{d})}\bigg\|\frac{w}{r}\bigg\|_{L^{\ift}(\bbR^{d})}+m\nrm{u^{s}}_{L^{\ift}(\bbR^{d})}\bigg\|\frac{w}{s}\bigg\|_{L^{\ift}(\bbR^{d})}\\
			&\leq n\nrm{u^{r}}_{L^{\ift}(\bbR^{d})}\bigg\|\frac{w}{r^{n}}\bigg\|_{L^{\ift}(\bbR^{d})}^{1/n}\nrm{w}_{L^{\ift}(\bbR^{d})}^{1-1/n}+m\nrm{u^{s}}_{L^{\ift}(\bbR^{d})}\bigg\|\frac{w}{s^{m}}\bigg\|_{L^{\ift}(\bbR^{d})}^{1/m}\nrm{w}_{L^{\ift}(\bbR^{d})}^{1-1/m}\\
			& {\lesssim_{d,w_{0}} nL^{d-1}\bigg\|\frac{w}{r^{n}}\bigg\|_{L^{\ift}(\bbR^{d})}^{1/n}\nrm{w}_{L^{\ift}(\bbR^{d})}^{1-1/n}+mL^{d-1}\bigg\|\frac{w}{s^{m}}\bigg\|_{L^{\ift}(\bbR^{d})}^{1/m}\nrm{w}_{L^{\ift}(\bbR^{d})}^{1-1/m}}.\label{eq:ddtwift}
		\end{align}
		First, assume that we have \eqref{eq:blowupL} for some $0<T^{\ast}<+\ift$. Then this implies that we have
		\begin{equation}\label{eq:Ltift}
			\lim\limits_{t \nearrow T^{\ast}}L(t)=+\ift.
		\end{equation}
		Then we use the relation
		$$ \dot{L}(t)\leq\nrm{u(t)}_{L^{\ift}(\bbR^{d})}\lesssim_{d}
		L(t)\nrm{w(t)}_{L^{\ift}(\bbR^{d})}, $$
		which is obtained from \eqref{eq:uLd-1}. From this, we have
		$$ \frac{d}{dt}(\ln L(t))\lesssim_{d}
		\nrm{w(t)}_{L^{\ift}(\bbR^{d})}, $$
		which gives us
		\begin{equation}\label{eq:Ltwtrel}
			L(t)\lesssim_{d}
			L(0)\exp\bigg(\int_{0}^{t}\nrm{w(\tau)}_{L^{\ift}(\bbR^{d})}d\tau\bigg).
		\end{equation}
		Using this, \eqref{eq:Ltift} implies that we have
		$$ \int_{0}^{T^{\ast}}\nrm{w(t)}_{L^{\ift}(\bbR^{d})}dt=+\ift. $$
		Thus, we get
		\begin{equation}\label{eq:blowupTast}
			\lim\limits_{t \nearrow T^{\ast}}\nrm{w(t)}_{L^{\ift}(\bbR^{d})}=+\ift.
		\end{equation}
		Now to prove the other direction, assume that for any $0<T^{\ast}<+\ift$, we have
		\begin{equation}\label{eq:blowupnegation}
			\int_{0}^{T^{\ast}}L(t)^{d-1}dt
			<+\ift.
		\end{equation}
		Then this and the inequalities \eqref{eq:r-nwsol} and \eqref{eq:s-mwsol}, which are obtained by solving the differential inequalities \eqref{eq:r-nw} and \eqref{eq:s-mw}, imply that the norms $\nrm{w/r^{n}
		}_{L^{\ift}(\bbR^{d})}$ and $\nrm{w/s^{m}
		}_{L^{\ift}(\bbR^{d})}$ are bounded at $t=T^{\ast}$. Lastly, from the boundedness of these norms, the assumption \eqref{eq:blowupnegation}, and the differential inequality \eqref{eq:ddtwift}, we have that the norm $\nrm{w}_{L^{\ift}(\bbR^{d})}$ is bounded at $t=T^{\ast}$.
	\end{proof}
	
	\begin{proof}[Proof of Theorem \ref{thm:0}]
		We let the initial data $w_{0}$ as in \eqref{eq:patch}. Then by Proposition \ref{prop:lwp}, we have the local-in-time regularity of the solution $w$ of the form \eqref{eq:patchsol}. Additionally, \eqref{eq:blowupL} implies
		\begin{equation}\label{eq:Lblowup}
			\sup_{t\in[0,T^{\ast})}L(t)=+\ift.
		\end{equation} 
			On the other hand, if we have \eqref{eq:Lblowup} for some $0<T^{\ast}<+\ift$, then we have \eqref{eq:Ltift}. Then we can repeat the same argument from the proof of Proposition \ref{prop:lwp} to obtain the blowup \eqref{eq:blowupTast} at time $t=T^{\ast}$. 
	\end{proof}
	
	

	\section{Support growth} \label{sec:growth}
	In this section, we present the proof of Theorem \ref{thm:1}. Throughout this section, we denote $T_{\text{max}}\in (0,+\ift]$ as the maximal time of existence of the solution $w$ of \eqref{eq:Eulereq}. 
	First, we define impulses of the solution $w$ in $r$ and 
	$s$-direction, respectively. 
	
	\begin{definition}\label{def:imp}
		For any $t\in [0,T_{\text{max}})$ the $r$ and $s$-impulses of $w$ 
		are defined as
		\begin{equation}\label{eq:imp}
			P^{r}(t):=\frac{1}{n+1}\iint_{\Pi}r^{n+1}w(t,r,s)drds,\quad P^{s}(t):=\frac{1}{m+1}\iint_{\Pi}s^{m+1}w(t,r,s)drds.
		\end{equation}
	\end{definition}
	The following lemma says that the time derivatives of the impulses above are positive 
	and negative
	, respectively.
	\begin{lem}\label{lem:Prtdot}
		For any $t\in [0,T_{\text{max}})$, we have
		\begin{align}
			\dot{P}^{r}(t)&=m\iint_{\Pi}\frac{r^{n}}{s}[u^{s}(t,r,s)]^{2}drds
			+\frac{1}{2}\int_{\lbrace s=0\rbrace}r^{n}[u^{r}(t,r,0)]^{2}dr,\label{eq:Prtdot}\\
			\dot{P}^{s}(t)&=-n\iint_{\Pi}\frac{s^{m}}{r}[u^{r}(t,r,s)]^{2}drds
			-\frac{1}{2}\int_{\lbrace r=0\rbrace}s^{m}[u^{s}(t,0,s)]^{2}ds.\label{eq:Pstdot}
		\end{align}
	\end{lem}
	This implies that the $r$-impulse $P^{r}$ is non-decreasing and the $s$-impulse $P^{s}$ is non-increasing in time.
	
	\begin{proof}
		Using the dynamics of $w$, the time derivative of $P^{r}$ can be written as
		\begin{equation}
			\begin{split}
				\dot{P}^{r}(t)&=\frac{1}{n+1}\frac{d}{dt}\iint_{\Pi}(r^{2n+1}s^{m})\frac{w
				}{r^{n}s^{m}}drds =-\frac{1}{n+1}\iint_{\Pi}(r^{2n+1}s^{m})(u^{r}\rd_{r}+u^{s}\rd_{s})\bigg(\frac{w
				}{r^{n}s^{m}}\bigg)drds.
			\end{split}
		\end{equation}
		Then using integration by parts, we have
		\begin{equation}
			\begin{split}
				\dot{P}^{r}(t)
		&=-\frac{1}{n+1}\bigg[0-0
		-(2n+1)\iint_{\Pi}r^{n}u^{r}wdrds-\iint_{\Pi}r^{n+1}\rd_{r}u^{r}wdrds\\
		&\qquad\qquad\quad+0-0-m\iint_{\Pi}\frac{r^{n+1}}{s}u^{s}wdrds-\iint_{\Pi}r^{n+1}\rd_{s}u^{s}wdrds\bigg] =\iint_{\Pi}r^{n}u^{r}wdrds.
	\end{split}
\end{equation}
The divergence-free condition \eqref{eq:divergencefree} was used in the second equality. Now we plug in $w=\rd_{r}u^{s}-\rd_{s}u^{r}$ and use integration by parts to get
\begin{equation}
	\begin{split}
		\dot{P}^{r}(t)&=\iint_{\Pi}r^{n}u^{r}(\rd_{r}u^{s}-\rd_{s}u^{r})drds = -\iint_{\Pi}r^{n}\bigg[\frac{n}{r}u^{r}+\rd_{r}u^{r}\bigg]
		u^{s}drds
		-\iint_{\Pi}r^{n}\rd_{s}\bigg[\frac{(u^{r})^{2}}{2}\bigg]drds\\
		&=\iint_{\Pi}r^{n}\bigg[\frac{m}{s}u^{s}+\rd_{s}u^{s}\bigg]u^{s}drds-0+\int_{\lbrace s=0\rbrace}r^{n}\frac{(u^{r})^{2}}{2}dr+0.\\
	\end{split}
\end{equation}
In the third equality, we used the divergence-free condition \eqref{eq:divergencefree} once more. Finally, we use integration by parts one more time to obtain
\begin{equation}
	\begin{split}
		\dot{P}^{r}(t)
		&=m\iint_{\Pi}\frac{r^{n}}{s}(u^{s})^{2}drds+\iint_{\Pi}r^{n}\rd_{s}\bigg[\frac{(u^{s})^{2}}{2}\bigg]drds+\int_{\lbrace s=0\rbrace}r^{n}\frac{(u^{r})^{2}}{2}dr \\
		&=m\iint_{\Pi}\frac{r^{n}}{s}(u^{s})^{2}drds+\int_{\lbrace s=0\rbrace}r^{n}\frac{(u^{r})^{2}}{2}dr.
	\end{split}
\end{equation}
The time derivative of $P^{s}$ can be computed in the same way.
\end{proof}

\begin{lem}\label{lem:lowerbd}
For any $r\geq0$, $\rd_{s}\psi(r,0)$ has the lower bound
\begin{equation}\label{eq:rdspsir0bd}
	\rd_{s}\psi(r,0)\geq c_{d}r\iint_{\Pi}\frac{\br{r}^{n+1}\br{s}^{m+1}}{[(r+\br{r})^{2}+\br{s}^{2}]^{d/2+1}}w(\br{r},\br{s})d\br{r}d\br{s}.
\end{equation}
\end{lem}

\begin{proof}
The integral form of $\rd_{s}\psi$ when $s=0$ is computed as
\begin{equation}
	\begin{split}
		&\rd_{s}\psi(r,0)=c_{d}\iint_{\Pi}\int_{0}^{\pi}\int_{0}^{\pi}\frac{\br{r}^{n}\br{s}^{m+1}\sin^{n-1}\br{\tht}_{1}\cos\br{\tht}_{1}\sin^{m-1}\br{\phi}_{1}\cos^{2}\br{\phi}_{1}}{(r^{2}+\br{r}^{2}-2r\br{r}\cos\br{\tht}_{1}+\br{s}^{2})^{d/2-1}}w(\br{r},\br{s})d\br{\tht}_{1}d\br{\phi}_{1}d\br{r}d\br{s}\\
		&\quad=c_{d}\iint_{\Pi}\int_{0}^{\pi/2}\int_{0}^{\pi/2}\br{r}^{n}\br{s}^{m+1}\sin^{n-1}\br{\tht}_{1}\cos\br{\tht}_{1}\sin^{m-1}\br{\phi}_{1}\cos^{2}\br{\phi}_{1}\\
		&\qquad\cdot\bigg[\frac{1}{(r^{2}+\br{r}^{2}-2r\br{r}\cos\br{\tht}_{1}+\br{s}^{2})^{d/2-1}}-\frac{1}{(r^{2}+\br{r}^{2}+2r\br{r}\cos\br{\tht}_{1}+\br{s}^{2})^{d/2-1}}\bigg]w(\br{r},\br{s})d\br{\tht}_{1}d\br{\phi}_{1}d\br{r}d\br{s}.
	\end{split}
\end{equation}

For simplicity, we denote $A':=r^{2}+\br{r}^{2}+\br{s}^{2}$ and $B'=2r\br{r}\cos\br{\tht}_{1}$. 
Then using the inequality
\begin{equation}
\begin{split}
	&\frac{1}{(A'-B')^{d/2}}-\frac{1}{(A'+B')^{d/2}}\\
	&\qquad=\frac{2B'[(A'+B')^{d-1}+(A'+B')^{d-2}(A'-B')+\cdots+(A'+B')(A'-B')^{d-2}+(A'-B')^{d-1}]}{(A'-B')^{d/2}(A'+B')^{d/2}[(A'-B')^{d/2}+(A'+B')^{d/2}]}\\
	&\qquad\geq \frac{B'}{(A'-B')^{d/2}(A'+B')}\geq\frac{B'}{(A'+B')^{d/2+1}} \geq\frac{B'}{[(r+\br{r})^{2}+\br{s}^{2}]^{d/2+1}},
\end{split}
\end{equation}
we have the lower bound in \eqref{eq:rdspsir0bd}:
\begin{equation}
	\begin{split}
		\rd_{s}\psi(r,0)&\gtrsim_{d} r\iint_{\Pi}\int_{0}^{\pi/2}\int_{0}^{\pi/2}\frac{\br{r}^{n+1}\br{s}^{m+1}\sin^{n-1}\br{\tht}_{1}\cos^{2}\br{\tht}_{1}\sin^{m-1}\br{\phi}_{1}\cos^{2}\br{\phi}_{1}}{[(r+\br{r})^{2}+\br{s}^{2}]^{d/2+1}}w(\br{r},\br{s})d\br{\tht}_{1}d\br{\phi}_{1}d\br{r}d\br{s}\\
		&\gtrsim_{d}r\iint_{\Pi}\frac{\br{r}^{n+1}\br{s}^{m+1}}{[(r+\br{r})^{2}+\br{s}^{2}]^{d/2+1}}w(\br{r},\br{s})d\br{r}d\br{s}.
	\end{split}
\end{equation}
This finishes the proof. 
\end{proof}

\begin{proof}[Proof of Theorem \ref{thm:1}]
	{Since the conclusion \eqref{eq:patchstretching} is immediate for the case $T_{\text{max}}<+\ift$ due to the blowup criterion \eqref{eq:blowupcrit} from Theorem \ref{thm:0}, we assume that we have $T_{\text{max}}=+\ift$.} 
	
	Suppose that there exists $0<M<\ift$ such that we have
	\begin{equation}\label{eq:suppbdd}
		\sup_{t\geq 0
		} L(t)
		\leq M.
	\end{equation}
	{Then since $w(t,\cdot)$ is supported in the region $\lbrace |x|\leq L(t)\rbrace$, we have for any $t\geq 0$, 
		\begin{equation}\label{eq:Prtbt}
			\begin{split}
				P^{r}(t)&\leq \frac{L(t)^{n+1}}{n+1}\iint_{\Pi}w(t,r,s)drds=\frac{L(t)^{n+1}}{n+1}\iint_{\Pi}w_{0}(r,s)drds \leq \frac{M^{n+1}}{n+1}\iint_{\Pi}w_{0}(r,s)drds.
			\end{split}
		\end{equation}
		For convenience, let us denote $\bt:=\frac{M^{n+1}}{n+1}\iint_{\Pi}w_{0}(r,s)drds>0$. 
	}
	\medskip
	
	\medskip

	First, recall from the equality \eqref{eq:Prtdot} in Lemma \ref{lem:Prtdot} that we have
	\begin{equation}
		\begin{split}
			\dot{P}^{r}(t)&=m\iint_{\Pi}\frac{r^{n}}{s}[u^{s}(t,r,s)]^{2}drds
			+\frac{1}{2}\int_{0}^{\ift}r^{n}[u^{r}(t,r,0)]^{2}dr \gtrsim 
			\int_{0}^{\ift}r^{n}[u^{r}(t,r,0)]^{2}dr.
		\end{split}
	\end{equation}
	Then we use the fact
	$$ u^{r}(t,r,0)=-\bigg(\frac{m}{s}\psi(t,r,s)\bigg|_{s=0}+\rd_{s}\psi(t,r,0)\bigg)=-(m+1)\rd_{s}\psi(t,r,0) $$
	to get
	\begin{equation}\label{eq:dotPrlwrbd}
		\dot{P}^{r}(t)\gtrsim \int_{0}^{\ift}r^{n}[u^{r}(t,r,0)]^{2}dr \gtrsim_{d} \int_{0}^{\ift}r^{n}[\rd_{s}\psi(t,r,0)
		]^{2}dr.
	\end{equation}
	Now let us focus on the term $\rd_{s}\psi(t,r,0)$: we use the lower bound \eqref{eq:rdspsir0bd} from Lemma \ref{lem:lowerbd} to get 
	\begin{equation}
		\begin{split}
			\rd_{s}\psi(t,r,0)&\gtrsim_{d} 
			r\iint_{\Pi}\frac{\br{r}^{n+1}\br{s}^{m+1}}{[(r+\br{r})^{2}+\br{s}^{2}]^{d/2+1}}w(t,\br{r},\br{s})d\br{r}d\br{s} =r\iint_{\Omg_{t}}\frac{\br{r}^{2n+1}\br{s}^{2m+1}
			}{[(r+\br{r})^{2}+\br{s}^{2}]^{d/2+1}}
			d\br{r}d\br{s}.
		\end{split}
	\end{equation}
	Here, let us make the following claim:
	\medskip
	
	\textbf{Claim.} There exists $\alp>0$, independent of $t$, that satisfies
	\begin{equation}\label{eq:lowerbd1/2}
		\iint_{E_{t}}r^{n}s^{m}drds\geq\frac{1}{2}\iint_{\Omg_{0}}r^{n}s^{m}drds,
	\end{equation}
	where $E_{t}:=\Omg_{t}\cap [[0,M]\times [0,\alp]]$.
	\medskip
	
	For now, we assume that \textbf{Claim} is true and present its proof later. Then we can get
	\begin{equation}
		\begin{split}
			\rd_{s}\psi(t,r,0)&\gtrsim_{d} r\iint_{\Omg_{t}}\frac{\br{r}^{2n+1}\br{s}^{2m+1}
			}{[(r+\br{r})^{2}+\br{s}^{2}]^{d/2+1}}
			d\br{r}d\br{s} \geq r\iint_{E_{t}}\frac{\br{r}^{2n+1}\br{s}^{2m+1}}{[(r+\br{r})^{2}+\br{s}^{2}]^{d/2+1}}.
		\end{split}
	\end{equation}
	Then using the fact $\br{r}\leq M$ and $\br{s}\leq \alp$ for any $(\br{r},\br{s})\in E_{t}$, we have
	\begin{equation}
		\begin{split}
			\rd_{s}\psi(t,r,0)&\gtrsim_{d}  r\iint_{E_{t}}\frac{\br{r}^{2n+1}\br{s}^{2m+1}}{[(r+\br{r})^{2}+\br{s}^{2}]^{d/2+1}}d\br{r}d\br{s} \geq \frac{r}{[(r+M)^{2}+\alp^{2}]^{d/2+1}}\iint_{E_{t}}\br{r}^{2n+1}\br{s}^{2m+1}d\br{r}d\br{s}.
		\end{split}
	\end{equation}
	Now, we denote $ \eta_{t}:=(4|E_{t}|/\pi)^{1/2} $ and $B_{\eta_{t}}:=\lbrace (\br{r},\br{s})\in \Pi : \br{r}^{2}+\br{s}^{2}<\eta_{t}^{2}\rbrace$. Then we get
	\begin{equation}
		\begin{split}
			\rd_{s}\psi(t,r,0)&\gtrsim_{d} \frac{r}{[(r+M)^{2}+\alp^{2}]^{d/2+1}}\iint_{E_{t}}\br{r}^{2n+1}\br{s}^{2m+1}d\br{r}d\br{s}\\
			&\geq \frac{r}{[(r+M)^{2}+\alp^{2}]^{d/2+1}}\iint_{B_{\eta_{t}}}\br{r}^{2n+1}\br{s}^{2m+1}d\br{r}d\br{s} \simeq_{d} \frac{r}{[(r+M)^{2}+\alp^{2}]^{d/2+1}}\eta_{t}^{2d}.
		\end{split}
	\end{equation}
	Note that we have
	\begin{equation}
		\begin{split}
			\frac{\pi}{4}\eta_{t}^{2}&=|E_{t}|=\iint_{E_{t}}
			d\br{r}d\br{s} =\iint_{E_{t}}\frac{\br{r}^{n}\br{s}^{m}}{\br{r}^{n}\br{s}^{m}}d\br{r}d\br{s} \\
			&\geq \frac{1}{M^{n}\alp^{m}}\iint_{E_{t}}\br{r}^{n}\br{s}^{m}d\br{r}d\br{s} \geq \frac{1}{2M^{n}\alp^{m}}\iint_{\Omg_{0}}\br{r}^{n}\br{s}^{m}d\br{r}d\br{s}.
		\end{split}
	\end{equation}
	In the second inequality above, we used the estimate \eqref{eq:lowerbd1/2} from \textbf{Claim}. Hence, we obtain
	\begin{equation}
		\begin{split}
			\rd_{s}\psi(t,r,0)&\gtrsim_{d} \frac{r}{[(r+M)^{2}+\alp^{2}]^{d/2+1}}\eta_{t}^{2d}\\
			&\gtrsim_{d} \frac{r}{[(r+M)^{2}+\alp^{2}]^{d/2+1}}\bigg(\frac{1}{M^{n}\alp^{m}}\iint_{\Omg_{0}}\br{r}^{n}\br{s}^{m}d\br{r}d\br{s}\bigg)^{d}.
		\end{split}
	\end{equation}
	Plugging this estimate into the lower bound of $\dot{P}^{r}(t)$ in \eqref{eq:dotPrlwrbd}, we have
	\begin{equation}
		\begin{split}
			\dot{P}^{r}(t)&\gtrsim_{d} \int_{0}^{\ift}r^{n}[\rd_{s}\psi(t,r,0)]^{2}dr\\
			&\gtrsim_{d} \bigg(\frac{1}{M^{n}\alp^{m}}\iint_{\Omg_{0}}\br{r}^{n}\br{s}^{m}d\br{r}d\br{s}\bigg)^{2d}\int_{0}^{\ift}\frac{r^{n+2}}{[(r+M)^{2}+\alp^{2}]^{d+2}}dr>0.
		\end{split}
	\end{equation}
	From the above estimate, let us denote the last term, which is positive and independent of $t$, as $\br{c}>0$. Then if we take some time $T=2\bt/\br{c}>0$ where $\bt>0$ is from \eqref{eq:Prtbt}, we have
	\begin{equation}
		P^{r}(T)=P^{r}(0)+\int_{0}^{T}\dot{P}^{r}(\tau)d\tau>P^{r}(0)+\br{c}T=P^{r}(0)+2\bt,
	\end{equation}
	which contradicts to the bound \eqref{eq:Prtbt}.
	\medskip
	
	Now let us prove \textbf{Claim}: recall that the norm $\nrm{w}_{L^{1}(\Pi)}$ is conserved, which gives us
	$$ \iint_{\Omg_{t}}r^{n}s^{m}drds=\iint_{\Omg_{0}}r^{n}s^{m}drds.
	$$
	In addition, recall that the negativity $\dot{P}^{s}(t)<0$ implies that $P^{s}(t)$ is non-increasing:
	$$ \iint_{\Omg_{t}}r^{n}s^{2m+1}
	drds\leq \iint_{\Omg_{0}}r^{n}s^{2m+1}
	drds.
	$$
	Using these, we choose $\alp>0$ as
	$$ \alp=\bigg(\iint_{\Omg_{0}} r^{n}s^{2m+1}
	drds\bigg)^{1/(m+1)}\bigg(\frac{1}{2}\iint_{\Omg_{0}}r^{n}s^{m}
	drds\bigg)^{-1/(m+1)}. $$
	Then denoting $G_{t}:=\Omg_{t}\cap [[0,M]\times (\alp,\ift)]$, 
	we obtain
	\begin{equation*}
		\begin{split}
			\iint_{G_{t}
			}r^{n}s^{m}
			drds&\leq \iint_{G_{t}
			}\frac{r^{n}s^{2m+1}}{\alp^{m+1}}
			drds\leq\frac{1}{\alp^{m+1}}\iint_{\Omg_{t}
			}r^{n}s^{2m+1}
			drds\\
			&\leq\frac{1}{\alp^{m+1}}\iint_{\Omg_{0}
			}r^{n}s^{2m+1}
			drds =\frac{1}{2}\iint_{\Omg_{0}
			}r^{n}s^{m}
			drds.
		\end{split}
	\end{equation*}
	Since $E_{t}\cup G_{t}= \Omg_{t}$ and $E_{t}\cap G_{t}=\emptyset$, 
	this implies
	\begin{equation}
		\begin{split}
			\iint_{E_{t}}r^{n}s^{m}drds&=\iint_{\Omg_{t}}r^{n}s^{m}drds-\iint_{G_{t}}r^{n}s^{m}drds\\
			&=\iint_{\Omg_{0}}r^{n}s^{m}drds-\iint_{G_{t}}r^{n}s^{m}drds \geq \frac{1}{2}\iint_{\Omg_{0}}r^{n}s^{m}drds.
		\end{split}
	\end{equation}
	Therefore, the proof of \textbf{Claim} is done and this completes the proof of Theorem \ref{thm:1}. 
\end{proof}
 
\bibliographystyle{plain}
\bibliography{biblography-JL-250930}

\end{document}